\documentclass[12 pt]{article}
\usepackage{amsmath}
\usepackage{amssymb}
\usepackage{textcomp}
\usepackage{amsmath,amsthm,amsfonts,amssymb,amstext}
\usepackage{layout} 
\usepackage{xcolor} 
 
\newtheorem{thm}{Theorem}[section]
\newtheorem{cor}[thm]{Corollary}  

\newtheorem{prop}[thm]{Proposition}
\newtheorem{lemma}[thm]{Lemma}

\newtheorem{note}[thm]{Note}
\newtheorem{definition}[thm]{Definition}
\newtheorem{assump}[thm]{Assumption}
\renewenvironment{proof}{\medskip\noindent{\emph {Proof:}\ }}{\qed \medskip} 

\title{ \bf Some $q$-exponential formulas involving the double lowering operator $\psi$ for a tridiagonal pair}
\author{Sarah Bockting-Conrad\thanks{This research was partially supported by a grant from the College of Science and Health at DePaul University.}}
\date{} 

\begin{document}

\maketitle

\vspace{-.15 in}
\begin{abstract}
Let $\mathbb{K}$ denote an algebraically closed field and let $V$ denote a vector space over $\mathbb{K}$ with finite positive dimension.  
We consider an ordered pair of linear transformations $A: V\to V$ and $A^*: V \to V$ that satisfy the following four conditions:  
(i) Each of $A,A^*$ is diagonalizable;
(ii) there exists an ordering $\{ V_i\}_{i=0}^d$ of the eigenspaces of $A$ such that 
$A^* V_i \subseteq V_{i-1}+V_i+V_{i+1}
$ for $0\leq i\leq d$, where $V_{-1}=0$ and $V_{d+1}=0$;
(iii) there exists an ordering $\{V^*_i\}_{i=0}^{\delta}$ of the eigenspaces of $A^*$ such that 
$A V^*_i \subseteq V^*_{i-1}+V^*_i+V^*_{i+1}
$ for $0\leq i\leq \delta$, where $V^*_{-1}=0$ and $V^*_{\delta+1}=0$;
(iv) there does not exist a subspace $W$ of $V$ such that $AW \subseteq W$, $A^*W\subseteq W$, $W\neq 0$, $W\neq V$.  
We call such a pair a {\it tridiagonal pair} on $V$.  
We assume that $A,A^*$ belongs to a family of tridiagonal pairs said to have 
$q$-Racah type.  
Let $\{U_i\}_{i=0}^d$ and $\{U_i^\Downarrow\}_{i=0}^{d}$ denote the first and second split decompositions of $V$.  
In an earlier paper we introduced a double lowering operator $\psi:V\to V$ with the notable feature that both 
$\psi U_i\subseteq U_{i-1}$ 
and 
$\psi U_i^\Downarrow\subseteq U_{i-1}^\Downarrow$ for $0\leq i\leq d$, where $U_{-1}=0$ and $U_{-1}^\Downarrow=0$.  
In the same paper, we showed that there exists a unique linear transformation $\Delta:V\to V$ such that $\Delta(U_i)\subseteq U_i^{\Downarrow}$ and 
$(\Delta -I)U_i\subseteq U_0+U_1+\cdots +U_{i-1}$
for $0\leq i \leq d$.  
In the present paper, we show that $\Delta$ can be expressed as a product of two linear transformations; one is a $q$-exponential in $\psi$ and the other is a $q^{-1}$-exponential in $\psi$. We view $\Delta$ as a transition matrix from the first split decomposition of $V$ to the second. Consequently, we view the $q^{-1}$-exponential in $\psi$ as a transition matrix from the first split decomposition to a decomposition of $V$ which we interpret as a kind of halfway point.
This halfway point turns out to be the eigenspace decomposition of a certain linear transformation $\mathcal{M}$. We discuss the eigenspace decomposition of $\mathcal{M}$ and give the actions of various operators on this decomposition.

\end{abstract}

\section{Introduction}

Throughout this paper, $\mathbb{K}$ denotes an algebraically closed field.
We begin by recalling the notion of a tridiagonal pair.
We will use the following terms.  Let $V$ denote a vector space over $\mathbb{K}$ with finite positive dimension.  For a linear transformation $A:V\to V$ and a subspace $W\subseteq V$, we say that $W$ is an {\it eigenspace} of $A$ whenever $W\neq 0$ and there exists $\theta\in\mathbb{K}$ such that $W=\{ v\in V | Av=\theta v\}$.  In this case, $\theta$ is called the {\it eigenvalue} of $A$ associated with $W$.  We say that $A$ is {\it diagonalizable} whenever $V$ is spanned by the eigenspaces of $A$.

\begin{definition}  {\rm \cite[Definition 1.1]{Somealg}}. \label{def:tdp}
{\rm Let $V$ denote a vector space over $\mathbb{K}$ with finite positive dimension. 
By a {\em tridiagonal pair} (or {\em TD pair}) on $V$ we mean an ordered pair of linear 
transformations $A:V \to V$ and $A^*:V \to V$ that satisfy the following 
four conditions.
\begin{enumerate}
\item[{\rm (i)}] 
Each of $A,A^*$ is diagonalizable.
\item[{\rm (ii)}] 
There exists an ordering $\{V_i\}_{i=0}^d$ of the eigenspaces of $A$ 
such that 
\begin{equation}               \label{eq:t1}
A^* V_i \subseteq V_{i-1} + V_i+ V_{i+1} \qquad \qquad (0 \leq i \leq d),
\end{equation}
where $V_{-1} = 0$ and $V_{d+1}= 0$.
\item[{\rm (iii)}]
There exists an ordering $\{V^*_i\}_{i=0}^{\delta}$ of the eigenspaces of 
$A^*$ such that 
\begin{equation}                \label{eq:t2}
A V^*_i \subseteq V^*_{i-1} + V^*_i+ V^*_{i+1} 
\qquad \qquad (0 \leq i \leq \delta),
\end{equation}
where $V^*_{-1} = 0$ and $V^*_{\delta+1}= 0$.
\item [{\rm (iv)}]
There does not exist a subspace $W$ of $V$ such  that $AW\subseteq W$,
$A^*W\subseteq W$, $W\not=0$, $W\not=V$.
\end{enumerate}
We say the pair $A,A^*$ is {\em over $\mathbb{K}$}.}
\end{definition}

\begin{note}   \label{note:star}        \samepage
{\rm According to a common notational convention $A^*$ denotes 
the conjugate-transpose of $A$. We are not using this convention.
In a TD pair $A,A^*$ the linear transformations $A$ and $A^*$
are arbitrary subject to (i)--(iv) above.}
\end{note}

Referring to the TD pair in Definition \ref{def:tdp}, by  \cite[Lemma 4.5]{Somealg} 
the scalars 
 $d$ and $\delta $ are equal.
We call this common value the {\it diameter} of $A,A^*$.
To avoid trivialities, throughout this paper we assume that the diameter is at least one.\\

TD pairs first arose in the study of $Q$-polynomial distance-regular graphs and provided a way to study the irreducible modules of the Terwilliger algebra associated with such a graph.  Since their introduction, TD pairs have been found to appear naturally in a variety of other contexts including representation theory \cite{TDfamily, neubauer, td-uqsl2, qtetalgebra, IT:qRacah, Koornwinder, Koornwinder2, aw}, orthogonal polynomials \cite{Askeyscheme, 2LT-PA}, partially ordered sets \cite{LPintro}, statistical mechanical models \cite{baseilhac, dolangrady, Onsager}, and other areas of physics \cite{LPcm, odake}.  As a result, TD pairs have become an area of interest in their own right.  
Among the above papers on representation theory, there are several works that connect TD pairs to quantum groups \cite{TDfamily, bockting2, neubauer, td-uqsl2,   IT:qRacah}.  
These papers consider certain special classes of TD pairs.  
We call particular attention to \cite{bockting2}, in which the present author describes a new relationship between TD pairs in the $q$-Racah class and quantum groups.  The present paper builds off of this work. \\

In the present paper, we give a new relationship between the maps $\Delta,\psi:V\to V$ introduced in \cite{bockting1}, as well as 
describe a new decomposition of the underlying vector space that, in some sense, lies between the first and second split decompositions associated with a TD pair.
In order to motivate our results, we now recall some basic facts concerning TD pairs.  
For the rest of this section, let $A,A^*$ denote a TD pair on $V$, as in Definition \ref{def:tdp}.  Fix an ordering $\{V_i\}_{i=0}^d$ (resp. $\{V_i^*\}_{i=0}^d$) of the eigenspaces of $A$ (resp. $A^*$) which satisfies (\ref{eq:t1}) (resp. (\ref{eq:t2})).  For $0\leq i\leq d$ let $\theta_i$ (resp. $\theta_i^*$) denote the eigenvalue of $A$ (resp. $A^*$) corresponding to $V_i$ (resp. $V_i^*$).   
By \cite[Theorem 11.1]{Somealg} the ratios
\begin{equation*}
\frac{\theta_{i-2}-\theta_{i+1}}{\theta_{i-1}-\theta_i},\qquad\qquad \frac{\theta_{i-2}^*-\theta_{i+1}^*}{\theta_{i-1}^*-\theta_i^*}
\end{equation*}
are equal and independent of $i$ for $2\leq i\leq d-1$.  This gives two recurrence relations, whose solutions can be written in closed form.  
There are several cases \cite[Theorem 11.2]{Somealg}.  
The most general case is called the $q$-Racah case \cite[Section 1]{IT:qRacah}.  We will discuss this case shortly.\\

We now recall the split decompositions of $V$ \cite{Somealg}.  
For $0\leq i\leq d$ define
\begin{align*}
U_i&= (V^*_0+V^*_1+\cdots + V^*_i)\cap (V_i+V_{i+1}+\cdots + V_d),\\
U_i^{\Downarrow} &= (V^*_0+V^*_1+\cdots + V^*_i)\cap (V_0+V_1+\cdots + V_{d-i}).
\end{align*}
By \cite[Theorem 4.6]{Somealg},
both the sums $V=\sum_{i=0}^d U_i$ and $V=\sum_{i=0}^d U_i^{\Downarrow}$ are direct.  We call $\{U_i\}_{i=0}^d$ (resp. $\{U_i^{\Downarrow}\}_{i=0}^d$) the first split decomposition (resp.  second split decomposition) of $V$.  
In \cite{Somealg}, the authors showed that $A,A^*$ act on the first and second split decomposition in a particularly attractive way.  This will be described in more detail in Section \ref{section:U}. \\

We now describe the $q$-Racah case.  
We say that the TD pair $A,A^*$ has {\it $q$-Racah type} whenever
 there exist nonzero scalars $q,a,b\in\mathbb{K}$ 
 such that $q^4\neq 1$ and 
\begin{equation*}
\theta_i=aq^{d-2i}+a^{-1}q^{2i-d},\qquad\qquad
\theta_i^*=bq^{d-2i}+b^{-1}q^{2i-d}
\end{equation*}
for $0\leq i\leq d$.
For the rest of this section assume that $A,A^*$ has $q$-Racah type.\\

We recall the maps $K$ and $B$ \cite[Section 1.1]{augTDalg}.  
Let $K:V\to V$ denote the linear transformation such that for $0\leq i \leq d$, $U_i$ is an eigenspace of $K$ with eigenvalue $q^{d-2i}$.  Let $B:V\to V$ denote the linear transformation such that for $0\leq i \leq d$, $U_i^{\Downarrow}$ is an eigenspace of $B$ with eigenvalue $q^{d-2i}$.  
The relationship between $K$ and $B$ is discussed in considerable detail in \cite{bockting2}.  \\

We now bring in the linear transformation $\Psi:V\to V$ \cite[Lemma 11.1]{bockting1}.  As in \cite{bockting2}, we work with the normalization $\psi=(q-q^{-1})(q^d-q^{-d})\Psi$.  
A key feature of $\psi$ is that by 
\cite[Lemma 11.2, Corollary 15.3]{bockting1}, 
\begin{equation*}
\psi U_i\subseteq U_{i-1},\qquad\qquad \psi U_i^\Downarrow\subseteq U_{i-1}^\Downarrow
\end{equation*}
for $1\leq i\leq d$ and both $\psi U_0=0$ and $\psi U_0^\Downarrow=0$.  In \cite{bockting2}, it is shown how $\psi$ is related to several maps, including the maps $K,B$, as well as the map $\Delta$ which we now recall.  By \cite[Lemma 9.5]{bockting1}, 
there exists a unique linear transformation $\Delta:V\to V$ such that
\begin{align*}
&\Delta U_i\subseteq U_i^{\Downarrow} &\ (0\leq i \leq d),\\
&(\Delta -I)U_i\subseteq U_0+U_1+\cdots +U_{i-1} &\ (0\leq i \leq d).
\end{align*}
In \cite[Theorem 17.1]{bockting1}, the present author showed that both
\begin{align*}
\Delta &= \sum_{i=0}^d \left(\prod_{j=1}^i \frac{aq^{j-1}-a^{-1}q^{1-j}}{q^j-q^{-j}}\right)\psi^i, \qquad\qquad
\Delta^{-1} &=\sum_{i=0}^d \left(\prod_{j=1}^i \frac{a^{-1}q^{j-1}-aq^{1-j}}{q^j-q^{-j}}\right)\psi^i. 
\end{align*}
The primary goal of this paper is to provide factorizations of these power series in $\psi$ and to investigate the consequences of these factorizations.  
We accomplish this goal using a linear transformation $\mathcal{M}:V\to V$ given by
\begin{equation*}
	\mathcal{M} =\frac{aK-a^{-1}B}{a-a^{-1}}.
\end{equation*}
By construction, $\mathcal{M}^\Downarrow=\mathcal{M}$.  One can quickly check that $\mathcal{M}$ is invertible. 
We show that the map $\mathcal{M}$ is equal to each of 
\begin{align*}
(I-a^{-1}q\psi)^{-1}K, \qquad K(I-a^{-1}q^{-1}\psi)^{-1}, \qquad (I-aq\psi)^{-1}B, \qquad B(I-aq^{-1}\psi)^{-1}.
\end{align*}
We give a number of different relations involving the maps $\mathcal{M},K,B,\psi$, the most significant of which are 
the following:
\begin{align*}
&K \exp_q\left(\frac{a^{-1}}{q-q^{-1}}\psi\right)=\exp_q\left(\frac{a^{-1}}{q-q^{-1}}\psi\right) \mathcal{M},\\
&B \exp_q\left(\frac{a}{q-q^{-1}}\psi\right)=\exp_q\left(\frac{a}{q-q^{-1}}\psi\right) \mathcal{M}.
\end{align*}
Using these equations, we obtain our main result which is that both
\begin{align*}
\Delta&=\exp_{q}\left(\frac{a}{q-q^{-1}}\psi\right)\exp_{q^{-1}}\left(-\frac{a^{-1}}{q-q^{-1}}\psi\right),\\
\Delta^{-1}&=\exp_{q}\left(\frac{a^{-1}}{q-q^{-1}}\psi\right)\exp_{q^{-1}}\left(-\frac{a}{q-q^{-1}}\psi\right).
\end{align*}

Due to its important role in the factorization of $\Delta$, we explore the map $\mathcal{M}$ further.  
We show that $\mathcal{M}$ is diagonalizable with eigenvalues $q^{d},q^{d-2},q^{d-4},\mathellipsis,q^{-d}$.  
For $0\leq i\leq d$, let $W_i$ denote the eigenspace of $\mathcal{M} $ corresponding to the eigenvalue $q^{d-2i}$.  We show that 
for $0\leq i\leq d$, 
\begin{align*}
&U_i=\exp_q\left(\frac{a^{-1}}{q-q^{-1}}\psi\right) W_i,&&U_i^\Downarrow=\exp_q\left(\frac{a}{q-q^{-1}}\psi\right) W_i,\\
&W_i=\exp_{q^{-1}}\left(-\frac{a^{-1}}{q-q^{-1}}\psi\right) U_i, 
&&W_i=\exp_{q^{-1}}\left(-\frac{a}{q-q^{-1}}\psi\right) U_i^\Downarrow.
\end{align*}
In light of this result, we interpret the decomposition $\{W_i\}_{i=0}^d$ as a sort of halfway point between the first and second split decompositions.  
We explore this decomposition further and give the actions of $\psi, K,B,\Delta,A,A^*$ on $\{W_i\}_{i=0}^d$.  
We then give the actions of $\mathcal{M}^{\pm 1}$ on $\{U_i\}_{i=0}^d,\{U_i^\Downarrow\}_{i=0}^d, \{V_i\}_{i=0}^d, \{V_i^*\}_{i=0}^d$.  We conclude the paper with a discussion of the special case when $A,A^*$ is a Leonard pair. \\

The present paper is organized as follows.  
In Section \ref{section:prelim} we discuss some preliminary facts concerning TD pairs and TD systems.  
In Sections \ref{section:U} we discuss the split decompositions of $V$ as well as the maps $K$ and $B$.  
In Section \ref{section:psi} we discuss the map $\psi$.
In Section \ref{section:delta} we recall the map $\Delta$ and give $\Delta$ as a power series in $\psi$. 
In Section \ref{section:M} we introduce the map $\mathcal{M}$ and describe its relationship with $A,K,B,\psi$. 
In Section \ref{section:factorization} we express $\Delta$ as a product of two linear transformations; one is a $q$-exponential in $\psi$ and the other is a $q^{-1}$-exponential in $\psi$.
In Section \ref{section:W} we describe the eigenvalues and eigenspaces of $\mathcal{M}$ and discuss how the eigenspace decomposition of $\mathcal{M}$ is related to the first and second split decompositions. 
In Section \ref{section:actionW} we discuss the actions of $\psi, K,B,\Delta,A,A^*$ on the eigenspace decomposition of $\mathcal{M}$. 
In Section \ref{section:Maction} we describe the action of $\mathcal{M} $ on the first and second split decompositions of $V$, as well as on the eigenspace decompositions of $A,A^*$.
In Section \ref{section:LS} we consider the case when $A,A^*$ is a Leonard pair.

\section{Preliminaries}\label{section:prelim}

When working with a tridiagonal pair, it is useful to consider a closely related object called a tridiagonal system. In order to define this object, we first recall some facts from elementary linear algebra \cite[Section 2]{Somealg}.\\

We use the following conventions.  When we discuss an algebra, we mean a unital associative algebra. When we discuss a subalgebra, we assume that it has the same unit as the parent algebra.\\

Let $V$ denote a vector space over $\mathbb{K}$ with finite positive dimension. 
By a {\it decomposition} of $V,$ we mean a sequence of nonzero subspaces whose direct sum is $V$.  
Let $ {\rm End} (V)$ denote the $\mathbb{K}$-algebra consisting of all linear transformations from $V$ to $V$.
Let $A$ denote a diagonalizable element in ${\rm End}(V)$.  
Let $\{V_i\}_{i=0}^d$ denote an ordering of the eigenspaces of $A$.
For $0\leq i\leq d$ let $\theta_i$ be the eigenvalue of $A$ corresponding to $V_i$.
Define $E_i\in {\rm End}(V)$ by
$ (E_i-I)V_i=0$ and 
$E_iV_j=0$ if $j\neq i$  $(0\leq j\leq d)$.
In other words, $E_i$ is the projection map from $V$ onto $V_i$.  
We refer to $E_i$ as the {\it primitive idempotent} of
$A$ associated with $\theta_i$.
By elementary linear algebra, 
(i) $AE_i = E_iA = \theta_iE_i$   $(0 \leq i \leq d)$;
(ii) $E_iE_j = \delta_{ij}E_i$   $(0 \leq i,j\leq d)$;
(iii) $V_i=E_iV$   $(0 \leq i \leq d)$;
(iv) $I=\sum_{i=0}^d E_i$.
Moreover 
\begin{align*}
E_i = \prod_{\genfrac{}{}{0 pt}{}{0 \leq  j \leq d}{j\not=i}} \frac{A-\theta_j I}{\theta_i-\theta_j}\qquad \qquad (0 \leq i \leq d).\label{EA} 
\end{align*}

Let $M$ denote the subalgebra of ${\rm End}(V)$ generated by $A$.  
Note that each of $\{ A^i\}_{i=0}^d$, $\{E_i\}_{i=0}^d$ is a basis for the $\mathbb{K}$-vector space $M$.\\  

Let $A,A^*$ denote a TD pair on $V$.  An ordering of the eigenspaces of $A$ (resp. $A^*$) is said to be {\it standard} whenever it satisfies (\ref{eq:t1}) (resp. (\ref{eq:t2})).  
Let $\{V_i\}_{i=0}^d$ denote a standard ordering of the eigenspaces of $A$.  By \cite[Lemma 2.4]{Somealg}, the ordering $\{V_{d-i}\}_{i=0}^d$ is standard and no further ordering of the eigenspaces of $A$ is standard.  A similar result holds for the eigenspaces of $A^*$.  An ordering of the primitive idempotents of $A$ (resp. $A^*$) is said to be {\it standard} whenever the corresponding ordering of the eigenspaces of $A$ (resp. $A^*$) is standard.

\begin{definition}{\rm \cite[Definition 2.1]{TDclass}.} \label{def:TDsys} {\rm
Let $V$ denote a vector space over $\mathbb{K}$ with finite positive dimension.  
By a {\it tridiagonal system} (or {\em TD system}) on $V,$  we mean a 
sequence 
\begin{equation*}
\Phi = (A; \{ E_i\}_{i=0}^d; A^*;\{ E_i^*\}_{i=0}^d)
\end{equation*}
 that satisfies  (i)--(iii) below.
\begin{enumerate}
\item[{\rm (i)}] $A, A^*$ is a tridiagonal pair on $V$. 
\item [{\rm (ii)}]$\{ E_i\}_{i=0}^d$ is a standard ordering of the primitive 
idempotents of $\;A$.
\item[{\rm (iii)}] $\{ E_i^*\}_{i=0}^d$ is a standard ordering of the primitive 
idempotents of $\;A^*$.
\end{enumerate}
We call $d$ the {\it diameter} of $\Phi$,  
and say $\Phi$ is {\it over } $\mathbb{K}$. 
 For notational convenience, set $E_{-1}=0$, $E_{d+1}=0$, $
E^*_{-1}=0$, $E^*_{d+1}=0$.}
\end{definition}

In Definition \ref{def:TDsys} we do not assume that the primitive idempotents $\{ E_i\}_{i=0}^d,\{ E_i^*\}_{i=0}^d$ all 
have rank 1.  A TD system for which each of these primitive idempotents has rank 1 is called a Leonard system \cite{2LT}.
The Leonard systems  are classified up to isomorphism \cite[Theorem 1.9]{2LT}.\\

For the rest of this paper, fix a TD system $\Phi$ on $V$ as in Definition \ref{def:TDsys}.
Our TD system $\Phi$ can be modified in a number of ways to get a new TD system \cite[Section 3]{Somealg}.  For example, the sequence
\begin{equation*}
\Phi^{\Downarrow} = (A;\{ E_{d-i}\}_{i=0}^d;A^*;\{ E_i^*\}_{i=0}^d)
\end{equation*}
is a TD system on $V$. 
Following  {\rm \cite[Section 3]{Somealg}}, we call $\Phi^{\Downarrow}$ the {\it second inversion} of $\Phi$.
When discussing $\Phi^{\Downarrow}$, we use the following notational convention.
For any object $f$ associated with $\Phi$, let $f^{\Downarrow}$ denote the corresponding object associated with $\Phi^{\Downarrow}$.

\begin{definition}\label{def:main}{\rm
For $0\leq i\leq d$ let $\theta_i$ (resp. $\theta_i^*$) denote the eigenvalue of $A$ (resp. $A^*$) associated with $E_i$ (resp. $E_i^*$).
We refer to $\{\theta_i\}_{i=0}^d$ (resp. $\{\theta_i^*\}_{i=0}^d$) as the {\it eigenvalue sequence} (resp. {\it dual eigenvalue sequence}) of $\Phi$.  
}\end{definition}

By construction $\{\theta_i\}_{i=0}^d$ are mutually distinct and $\{\theta_i^*\}_{i=0}^d$ are mutually distinct.
  By \cite[Theorem 11.1]{Somealg}, the scalars
\begin{equation*}
\frac{\theta_{i-2}-\theta_{i+1}}{\theta_{i-1}-\theta_i},\qquad\qquad \frac{\theta_{i-2}^*-\theta_{i+1}^*}{\theta_{i-1}^*-\theta_i^*}\label{eq:intro-ratio}
\end{equation*}
are equal and independent of $i$ for $2\leq i\leq d-1$.  
For this restriction, the solutions have been found in closed form \cite[Theorem 11.2]{Somealg}.  
The most general solution is called $q$-Racah \cite[Section 1]{IT:qRacah}.  This solution is described as follows.  

\begin{definition}\label{def:racah}
{\rm Let $\Phi$ denote a TD system on $V$ as in Definition \ref{def:TDsys}.  
We say that  $\Phi$ has \emph{$q$-Racah type} whenever there exist nonzero scalars $q,a,b\in\mathbb{K}$ such that 
such that $q^4\neq 1$ and 
\begin{equation}
\theta_i=aq^{d-2i}+a^{-1}q^{2i-d},\qquad\qquad
\theta_i^*=bq^{d-2i}+b^{-1}q^{2i-d}
\label{eq:theta-assump}
\end{equation}
for $0\leq i\leq d$.  }
\end{definition}

\begin{note}{\rm
Referring to Definition \ref{def:racah}, the scalars $q,a,b$ are not uniquely defined by $\Phi$.  If $q,a,b$ is one solution, then their inverses give another solution.  
}\end{note}

For the rest of the paper, we make the following assumption.
\begin{assump}\label{assump:main}
We assume that our TD system $\Phi$ has $q$-Racah type. We fix $q,a,b$ as in Definition \ref{def:racah}.
\end{assump}

\begin{lemma}{\rm \cite[Lemma 2.4]{bockting2}.}\label{note:main}
With reference to Assumption \ref{assump:main}, the following hold.
\begin{itemize}
\item[{\rm (i)}] Neither of $a^2$, $b^2$ is among $q^{2d-2},q^{2d-4},\mathellipsis, q^{2-2d}$.
\item[{\rm (ii)}] $q^{2i}\neq 1$ for $1\leq i\leq d$.
\end{itemize}
\end{lemma}
\begin{proof}
The result follows from the comment below Definition \ref{def:main}.
\end{proof}

\section{The first and second split decomposition of $V$}\label{section:U}

Recall the TD system $\Phi$ from Assumption \ref{assump:main}.  
In this section we consider two decompositions of $V$ associated with $\Phi$, called the first and second split decomposition.\\    

For $0\leq i\leq d$ define 
\begin{align*}
U_i = (E^*_0V+E^*_1V+\cdots + E^*_iV)\cap (E_iV+E_{i+1}V+\cdots + E_dV).
\end{align*}
For notational convenience, define $U_{-1}=0$ and $U_{d+1}=0$.
Note that for $0\leq i\leq d$, 
\begin{equation*}
U_i^{\Downarrow}=(E_0^*V+E_1^*V +\cdots + E_i^* V)\cap (E_0V+E_1V+\cdots + E_{d-i} V).
\end{equation*}
By \cite[Theorem 4.6]{Somealg},  
the sequence $\{U_i\}_{i=0}^d$  (resp. $\{U_i^\Downarrow\}_{i=0}^d$) is a decomposition of $V$.  Following \cite{Somealg}, we refer to $\{U_i\}_{i=0}^d$  (resp. $\{U_i^\Downarrow\}_{i=0}^d$) as the {\it first split decomposition} (resp. {\it second split decomposition}) of $V$ with respect to $\Phi$.  
By \cite[Corollary 5.7]{Somealg}, for $0\leq i\leq d$ the dimensions of
$E_iV$, $E_i^*V$, $U_i$, $U_i^\Downarrow$ coincide; we denote the common dimension by $\rho_i$.  
By \cite[Theorem 4.6]{Somealg}, 
\begin{align}
E_iV+E_{i+1}V+\cdots+E_dV&=U_{i}+U_{i+1}+\cdots+U_d,\label{eq:EU}\\
E_0V+E_1V+\cdots+E_iV&=U_{d-i}^\Downarrow+U_{d-i+1}^\Downarrow+\cdots+U_d^\Downarrow,\label{eq:EUdd}\\
E_0^*V+E_1^*V+\cdots E_i^*V=U_0+U_1&+\cdots +U_i=U_0^\Downarrow+U_1^\Downarrow+\cdots +U_i^\Downarrow. \label{eq:UUdd}
\end{align}
By \cite[Theorem 4.6]{Somealg}, $A$ and $A^*$ act on the first split decomposition in the following way: 
\begin{align*}
&(A-\theta_i I)U_i\subseteq U_{i+1} &(0\leq i \leq d-1 ), \qquad &(A-\theta_d I)U_d=0,\\
&(A^*-\theta_i^* I)U_i\subseteq U_{i-1} &(1\leq i\leq d ), \qquad &(A^*-\theta_0^* I)U_0=0.
\end{align*}
By \cite[Theorem 4.6]{Somealg}, $A$ and $A^*$ act on the second split decomposition in the following way:
\begin{align*}
&(A-\theta_{d-i} I)U_i^{\Downarrow}\subseteq U_{i+1}^{\Downarrow}&(0\leq i\leq d-1 ), \qquad &(A-\theta_0 I)U_d^{\Downarrow}=0,\\
&(A^*-\theta_i^* I)U_i^{\Downarrow}\subseteq U_{i-1}^{\Downarrow}&(1\leq i\leq d ), \qquad &(A^*-\theta_0^* I)U_0^{\Downarrow}=0.
\end{align*}

\begin{definition}\label{def:KB}{\rm \cite[Definitions 3.1 and 3.2]{bockting2}. Define $K,B\in {\rm End}(V)$ such that for $0\leq i\leq d$, $U_i$ (resp. $U_i^\Downarrow$) is the eigenspace of $K$ (resp. $B$) with eigenvalue $q^{d-2i}$.  In other words,
\begin{align}
(K-q^{d-2i}I)U_i=0, \qquad \qquad(B-q^{d-2i}I)U_i^\Downarrow=0 \qquad\qquad (0\leq i\leq d).\label{eq:Kdef}
\end{align}
Observe that $B=K^{\Downarrow}$.  
}\end{definition}

By construction each of $K, B$ is invertible and diagonalizable on $V$.  \\

We now describe how $K$ and $B$ act on the eigenspaces of the other one.

\begin{lemma}{\rm \cite[Lemma 3.3]{bockting2}.}\label{lemma:KUdd}
For $0\leq i\leq d$,
\begin{align}
(B-q^{d-2i}I)U_i&\subseteq U_0+U_1+\cdots +U_{i-1},\label{eq:BU}\\
(K-q^{d-2i}I)U_i^\Downarrow &\subseteq U_0^\Downarrow+U_1^\Downarrow+\cdots +U_{i-1}^\Downarrow.\label{eq:KUdd}
\end{align}
\end{lemma}

Next we describe how $A,K,B$ are related. 

\begin{lemma}{\rm \cite[Section 1.1]{augTDalg}}.\label{lemma:AKq-Weyl}
Both
\begin{equation}
\frac{qKA-q^{-1}AK}{q-q^{-1}}=aK^2+a^{-1}I,\qquad\qquad 
\frac{qBA-q^{-1}AB}{q-q^{-1}}=a^{-1}B^2+aI. \label{eq:AKq-Weyl}
\end{equation}
\end{lemma}

\begin{lemma}{\rm \cite[Theorem 9.9]{bockting2}}.\label{thm:KBquad}
We have
\begin{equation}
aK^2-\frac{a^{-1}q-aq^{-1}}{q-q^{-1}}\ KB-\frac{aq-a^{-1}q^{-1}}{q-q^{-1}}\ BK+a^{-1}B^2=0.\label{eq:KBquad1}
\end{equation}
\end{lemma}

\section{The linear transformation $\psi$}\label{section:psi}

We continue to discuss the situation of Assumption \ref{assump:main}.   
 In \cite[Section 11]{bockting1} we introduced an element $\Psi\in {\rm End}(V)$.  
In \cite{bockting2} we used the normalization $\psi=(q-q^{-1})(q^d-q^{-d})\Psi$.
In \cite[Theorem 9.8]{bockting2}, we showed that $\psi$ is equal to some rational expressions involving $K,B$.  We now recall this result. 
We start with a comment.

\begin{lemma}{\rm \cite[Lemma 9.7]{bockting2}.}\label{lemma:invertible1}
Each of the following is invertible:
\begin{align}
&aI-a^{-1}BK^{-1}, &a^{-1}I-aKB^{-1},\label{eq:invertible1a}\\
&aI-a^{-1}K^{-1}B, &a^{-1}I-aB^{-1}K.
\end{align}
\end{lemma}

\begin{lemma}{\rm \cite[Theorem 9.8]{bockting2}.} \label{thm:psiequations}
The following four expressions coincide:
\begin{align}
& \frac{I-BK^{-1}}{q(aI-a^{-1}BK^{-1})},  &&\frac{I-KB^{-1}}{q(a^{-1}I-aKB^{-1})},\label{eq:psiequations1}\\ \medskip
& \frac{q(I-K^{-1}B)}{aI-a^{-1}K^{-1}B}, &&\frac{q(I-B^{-1}K)}{a^{-1}I-aB^{-1}K}.\label{eq:psiequations2}
\end{align}
In {\rm (\ref{eq:psiequations1})}, {\rm (\ref{eq:psiequations2})} the denominators are invertible by Lemma \ref{lemma:invertible1}.
\end{lemma}

\begin{definition}{\rm \label{def:psi}
Define $\psi\in {\rm End}(V)$ to be the common value of the four expressions in Lemma \ref{thm:psiequations}.
}\end{definition}

We now recall some facts concerning $\psi$.  

\begin{lemma}{\rm\cite[Lemma 5.4]{bockting2}.}\label{lemma:Kpsi}
Both
\begin{equation}
K\psi =q^2\psi K, \qquad\qquad B\psi =q^2\psi B.\label{eq:Kpsi}
\end{equation}
\end{lemma}

\begin{lemma}{\rm \cite[Lemma 11.2, Corollary 15.3]{bockting1}}. \label{lemma:psiU}
We have 
\begin{equation}
\psi U_i\subseteq U_{i-1},\qquad\qquad \psi U_i^\Downarrow\subseteq U_{i-1}^\Downarrow\qquad\qquad (1\leq i\leq d)
\end{equation}
and also $\psi U_0=0$ and $\psi U_0^\Downarrow=0$.   
Moreover $\psi^{d+1}=0$.
\end{lemma}

In Lemma \ref{thm:psiequations} we obtained $\psi$ as a rational expression in $BK^{-1}$ or $K^{-1}B$.  Next we solve for $BK^{-1}$ and $K^{-1}B$ as a rational function in $\psi$.  In order to state the answer, we will need the following result. 

\begin{lemma}{\rm \cite[Lemma 9.2]{bockting2}}.\label{lemma:invertible2}
Each of the following is invertible:
\begin{equation}
I-aq\psi,\qquad I-a^{-1}q\psi, \qquad I-aq^{-1}\psi,\qquad I-a^{-1}q^{-1}\psi.\label{eq:invertible2}
\end{equation}
Their inverses are as follows:
\begin{align}
(I-aq\psi)^{-1}&=\sum_{i=0}^da^iq^i\psi^i,\qquad &(I-a^{-1}q\psi)^{-1}&=\sum_{i=0}^d a^{-i}q^i\psi^i,\label{eq:psiinv1'}\\
(I-aq^{-1}\psi)^{-1}&=\sum_{i=0}^d a^iq^{-i}\psi^i,\qquad &(I-a^{-1}q^{-1}\psi)^{-1}&=\sum_{i=0}^d a^{-i}q^{-i}\psi^i.\label{eq:psiinv4'}
\end{align}
\end{lemma}

The next result is an immediate consequence of Lemma \ref{thm:psiequations}, Definition \ref{def:psi}, and Lemma \ref{lemma:invertible2}.

\begin{thm}{\rm \cite[Theorem 9.4]{bockting2}.}\label{thm:BK}
The following hold:
\begin{align}
BK^{-1}&=\frac{I-aq\psi}{I-a^{-1}q\psi}, \qquad &&KB^{-1}=\frac{I-a^{-1}q\psi}{I-aq\psi},\label{eq:BK-1}\\
K^{-1}B&=\frac{I-aq^{-1}\psi}{I-a^{-1}q^{-1}\psi}, \qquad  &&B^{-1}K=\frac{I-a^{-1}q^{-1}\psi}{I-aq^{-1}\psi}.\label{eq:BK-2}
\end{align}
In {\rm (\ref{eq:BK-1})}, {\rm (\ref{eq:BK-2})} the denominators are invertible by Lemma \ref{lemma:invertible2}. 
\end{thm}

\begin{lemma}{\rm \cite[Equation (22)]{bockting2}}.\label{lemma:Apsi}
We have
\begin{equation}
\frac{\psi A-A\psi}{q-q^{-1}}=\left(I-aq\psi\right)K-\left(I-a^{-1}q^{-1}\psi\right)K^{-1}.
\end{equation}
\end{lemma}
\begin{proof}
This result is a reformulation of \cite[Equation (22)]{bockting2} using \cite[Equation (14)]{bockting2}.
\end{proof}

\section{The linear transformation $\Delta$}\label{section:delta} 

We continue to discuss the situation of Assumption \ref{assump:main}.   
 In \cite[Section 9]{bockting1} we introduced an invertible element $\Delta\in {\rm End}(V)$.  In \cite{bockting1} we showed that $\Delta,\psi$ commute and in fact both $\Delta,\Delta^{-1}$ are power series in $\psi$.   These power series will be the central focus of this paper. We will show that each of those power series factors as a product of two power series, each of which is a quantum exponential in $\psi$.

\begin{lemma}{\rm \cite[Lemma 9.5]{bockting1}.}\label{lemma:DD'}
There exists a unique $\Delta\in {\rm End}(V)$ such that
\begin{align}
&\Delta U_i\subseteq U_i^{\Downarrow} &\ (0\leq i \leq d),\label{eq:DeltaU21}\\
&(\Delta -I)U_i\subseteq U_0+U_1+\cdots +U_{i-1} &\ (0\leq i \leq d).\label{eq:DeltaU22}
\end{align}
\end{lemma}

\begin{lemma}{\rm \cite[Lemmas 9.3 and 9.6]{bockting1}.}\label{lemma:Delta^(-1)}
The map $\Delta$ is invertible.  Moreover
$\Delta^{-1}=\Delta^{\Downarrow}$ and
\begin{align}
(\Delta^{-1} -I)U_i\subseteq U_0+U_1+\cdots +U_{i-1} \qquad\qquad (0\leq i\leq d). \label{eq:Delta^{-1}U}
\end{align}
\end{lemma}

\begin{lemma}
The map $\Delta-I$ is nilpotent.
Moreover  $\Delta K=B\Delta$.
\end{lemma} 
\begin{proof}
The first assertion follows from \eqref{eq:DeltaU22}.
The last assertion follows from \eqref{eq:DeltaU21} and Definition \ref{def:KB}.
\end{proof}

The map $\Delta$ is characterized as follows.

\begin{lemma}{\rm \cite[Lemma 9.8]{bockting1}.}\label{lemma:DeltaEV}
The map $\Delta$ is the unique element of ${\rm End}(V)$ such that 
\begin{align}
&(\Delta -I)E_i^{*} V\subseteq E_0^{*}V+E_1^{*}V+\cdots +E_{i-1}^{*}V &\ (0\leq i \leq d),\label{Delta1}\\
&\Delta (E_iV+E_{i+1}V+\cdots+E_d V)= E_0 V+E_{1}V+\cdots+E_{d-i} V &(0\leq i \leq d). \label{Delta2}
\end{align}
\end{lemma} 

\begin{thm}{\rm \cite[Theorem 17.1]{bockting1}.}  \label{thm:Deltapoly}
Both
\begin{align}
\Delta &= \sum_{i=0}^d \left(\prod_{j=1}^i \frac{aq^{j-1}-a^{-1}q^{1-j}}{q^j-q^{-j}}\right)\psi^i, 
\label{eq:MT0}\\
\Delta^{-1} &=\sum_{i=0}^d \left(\prod_{j=1}^i \frac{a^{-1}q^{j-1}-aq^{1-j}}{q^j-q^{-j}}\right)\psi^i. 
\label{eq:MT00}
\end{align}
\end{thm}

In \eqref{eq:MT0} and \eqref{eq:MT00}, the elements $\Delta,\Delta^{-1}$ are expressed as a power series in $\psi$.  In the present paper, we factor these power series and interpret the results.  This interpretation will involve a linear transformation $\mathcal{M}$.  We introduce $\mathcal{M}$ in the next section.

\section{The linear transformation $\mathcal{M}$}\label{section:M}

We continue to discuss the situation of Assumption \ref{assump:main}.  In this section we introduce an element $\mathcal{M} \in{\rm End}(V)$. We explain how $\mathcal{M}$ is related to $K,B,\psi,A$.

\begin{definition}{\rm \label{def:M}
Define  $\mathcal{M} \in {\rm End}(V)$ by
\begin{equation}
\mathcal{M} =\frac{aK-a^{-1}B}{a-a^{-1}}.\label{eq:Mdef}
\end{equation}
}\end{definition}

By construction, $\mathcal{M}^\Downarrow=\mathcal{M} $.  
Evaluating \eqref{eq:Mdef} using Lemma \ref{lemma:invertible1}, we see that $\mathcal{M}$ is invertible.

\begin{lemma}\label{lemma:Mequal}
The map $\mathcal{M}$ is equal to each of:
\begin{align*}
(I-a^{-1}q\psi)^{-1}K, \qquad K(I-a^{-1}q^{-1}\psi)^{-1}, \qquad (I-aq\psi)^{-1}B, \qquad B(I-aq^{-1}\psi)^{-1}.
\end{align*}
\end{lemma}
\begin{proof}
We first show that $\mathcal{M}=(I-a^{-1}q\psi)^{-1}K$. By Definition \ref{def:M}, 
\begin{equation*}
(a-a^{-1})\mathcal{M}  K^{-1}=aI-a^{-1}BK^{-1}.
\end{equation*}
The result follows from this fact along with the equation on the left in \eqref{eq:BK-1}.

The remaining assertions follow from Theorem \ref{thm:BK}.
\end{proof}

Lemma \ref{lemma:Mequal} can be reformulated as follows.
\begin{lemma}\label{lemma:K-M} 
We have 
\begin{align}
K&=\left(I-a^{-1}q\psi\right)\mathcal{M}, \qquad &&K=\mathcal{M} \left(I-a^{-1}q^{-1}\psi\right), 
\label{eq:K-M}\\
B&=\left(I-aq\psi\right)\mathcal{M},\qquad &&B=\mathcal{M}\left(I-aq^{-1}\psi\right). \label{eq:B-M}
\end{align}
\end{lemma}

For later use, we give several descriptions of $\mathcal{M}^{\pm 1}$.

\begin{lemma}\label{lemma:Minv-exp}
The map
$\mathcal{M}^{-1}$ is equal to each of:
\begin{align*}
K^{-1}(I-a^{-1}q\psi), \qquad (I-a^{-1}q^{-1}\psi)K^{-1}, \qquad B^{-1}(I-aq\psi), \qquad (I-aq^{-1}\psi)B^{-1}.
\end{align*}
\end{lemma}
\begin{proof}
Immediate from Lemma \ref{lemma:Mequal}.
\end{proof}

\begin{lemma}\label{lemma:Msum}
The map
$\mathcal{M}$ is equal to each of:
\begin{equation}
	K\sum_{n=0}^d a^{-n}q^{-n}\psi^n,\qquad
	\sum_{n=0}^d a^{-n}q^{n}\psi^nK,\qquad
	B\sum_{n=0}^d a^{n}q^{-n}\psi^n,\qquad
	\sum_{n=0}^d a^{n}q^{n}\psi^n B
\end{equation}
\end{lemma} 
\begin{proof}
Use Lemma \ref{lemma:invertible2} and Lemma \ref{lemma:Mequal}.
\end{proof}

We now give some attractive equations that show how $\mathcal{M}$ is related to $\psi, K,B,A$.

\begin{lemma}\label{lemma:Mpsi}
We have 
\begin{equation}
\mathcal{M} \psi=q^2\psi \mathcal{M} .
\end{equation}
\end{lemma}
\begin{proof}
Use Lemma \ref{lemma:Kpsi} and Definition \ref{def:M}.  
\end{proof}

\begin{lemma}\label{lemma:qM^K}
We have
\begin{align}
\frac{q\mathcal{M}^{-1}K-q^{-1}K\mathcal{M}^{-1}}{q-q^{-1}}=I,\qquad\qquad \frac{q\mathcal{M}^{-1}B-q^{-1}B\mathcal{M}^{-1}}{q-q^{-1}}=I.\label{eq:MinvK}
\end{align}
\end{lemma}
\begin{proof}
Use Lemma \ref{lemma:Minv-exp}.
\end{proof}

\begin{lemma}\label{lemma:qAM^}
We have
\begin{align}
\frac{qA\mathcal{M}^{-1}-q^{-1}\mathcal{M}^{-1}A}{q-q^{-1}}=(a+a^{-1})I-(q+q^{-1})\psi.\label{eq:qAM^}
\end{align}
\end{lemma}
\begin{proof}
Use Lemma \ref{lemma:AKq-Weyl}, Lemma \ref{lemma:Kpsi}, Lemma \ref{lemma:Apsi}, and Lemma \ref{lemma:Minv-exp}. 
\end{proof}

\begin{lemma}\label{lemma:M^-2A}
We have
\begin{equation}
\mathcal{M}^{-2}A-(q^2+q^{-2})\mathcal{M}^{-1}A\mathcal{M}^{-1}+A\mathcal{M}^{-2}=-(q-q^{-1})^2(a+a^{-1})\mathcal{M}^{-1}.
\end{equation}
\end{lemma}
\begin{proof}
Use Lemma \ref{lemma:Mpsi} and Lemma \ref{lemma:qAM^}.
\end{proof}

\section{A factorization of $\Delta$}\label{section:factorization}

We continue to discuss the situation of Assumption \ref{assump:main}.   
We now bring in the $q$-exponential function \cite{gasper}. 
In \cite[Theorem 17.1]{bockting1} we expressed $\Delta$ as a power series in $\psi$.  In this section we strengthen this result in the following way.  We express $\Delta$ as a product of two linear transformations; one is a $q$-exponential in $\psi$ and the other is a $q^{-1}$-exponential in $\psi$.\\

For an integer $n$, define 
\begin{equation}
[n]_q=\frac{q^n-q^{-n}}{q-q^{-1}}
\end{equation}
and for $n\geq 0$, define
\begin{equation}
[n]_q^!=[n]_q[n-1]_q\cdots [1]_q.
\end{equation}
We interpret $[0]_q^!=1$.\\

We now recall the $q$-exponential function \cite{gasper}.  
For a nilpotent $T\in {\rm End}(V)$,
\begin{align}
\exp_q(T)=\sum_{n=0}^\infty \frac{q^{\binom{n}{2}}}{[n]_q^!} T^n.\label{eq:qexpdef}
\end{align}
The map $\exp_q(T)$ is invertible.  Its inverse is given by
\begin{align}
\exp_{q^{-1}}(-T)=\sum_{n=0}^\infty \frac{(-1)^n q^{-\binom{n}{2}}}{[n]_q^!} T^n.\label{eq:qexpinv}
\end{align}
Using \eqref{eq:qexpdef} we obtain
\begin{align}
(I-(q^2-1)T)\exp_q(q^2T)=\exp_q(T).\label{eq:qexp1}
\end{align}
For $S\in {\rm End}(V)$ such that $ST=q^2TS$,  we have 
\begin{align*}
S\exp_q(T)S^{-1}=\exp_q(STS^{-1})=\exp_q(q^2T).
\end{align*}
Consequently 
\begin{align}
S\exp_q(T)=\exp_q(q^2T)S.\label{eq:qexp2}
\end{align}
Combining \eqref{eq:qexp1} and \eqref{eq:qexp2},
\begin{align}
(I-(q^2-1)T)S\exp_q(T)=\exp_q(T)S.\label{eq:qexp3}
\end{align}

We return our attention to $K,B,\psi,\mathcal{M}$.

\begin{prop}\label{prop:KexpM}
Both
\begin{align}
&K \exp_q\left(\frac{a^{-1}}{q-q^{-1}}\psi\right)=\exp_q\left(\frac{a^{-1}}{q-q^{-1}}\psi\right) \mathcal{M} ,\label{eq:KexpM}\\
&B \exp_q\left(\frac{a}{q-q^{-1}}\psi\right)=\exp_q\left(\frac{a}{q-q^{-1}}\psi\right) \mathcal{M} .\label{eq:BexpM}
\end{align}
\end{prop}
\begin{proof}
Recall from Lemma \ref{lemma:Mpsi} that $\mathcal{M} \psi=q^2\psi \mathcal{M}$. 
We first obtain (\ref{eq:KexpM}).  To do this, in \eqref{eq:qexp3} take $S=\mathcal{M}$ and $T=\frac{a^{-1}}{q-q^{-1}}\psi$.  Evaluate the result using the equation $\mathcal{M}=(I-a^{-1}q\psi)^{-1}K$ from Lemma \ref{lemma:Mequal}.

Next we obtain (\ref{eq:BexpM}).  To do this, in \eqref{eq:qexp3} take $S=\mathcal{M}$ and $T=\frac{a}{q-q^{-1}}\psi$. Evaluate the result using the equation $\mathcal{M}=(I-aq\psi)^{-1}B$ from Lemma \ref{lemma:Mequal}.
\end{proof}

\medskip

The following is our main result.

\begin{thm}\label{thm:Deltaqexp}
Both
\begin{align}
\Delta&=\exp_{q}\left(\frac{a}{q-q^{-1}}\psi\right)\exp_{q^{-1}}\left(-\frac{a^{-1}}{q-q^{-1}}\psi\right),\label{eq:Delta}\\
\Delta^{-1}&=\exp_{q}\left(\frac{a^{-1}}{q-q^{-1}}\psi\right)\exp_{q^{-1}}\left(-\frac{a}{q-q^{-1}}\psi\right).\label{eq:Deltainv}
\end{align}
\end{thm}
\begin{proof}
We first show (\ref{eq:Delta}). Let $\tilde\Delta $ denote the expression on the right in  (\ref{eq:Delta}).  
Combining (\ref{eq:KexpM}) and (\ref{eq:BexpM}), we see that $\tilde\Delta K=B\tilde\Delta $.  Therefore $\tilde\Delta U_i=U_i^\Downarrow$ for $0\leq i\leq d$.
Observe that $\tilde\Delta -I$ is a polynomial in $\psi$ with zero constant term.  By Lemma \ref{lemma:psiU}, $(\tilde\Delta -I)U_i\subseteq U_0+U_1+\cdots+U_{i-1}$ for $0\leq i\leq d$.  
By Lemma \ref{lemma:DD'}, $\tilde\Delta =\Delta$.

To obtain (\ref{eq:Deltainv}) from (\ref{eq:Delta}), use \eqref{eq:qexpinv}.
\end{proof}

\begin{cor}\label{cor:prodsum}
We have
\begin{align*}
\exp_{q}\left(\frac{a}{q-q^{-1}}\psi\right)\exp_{q^{-1}}\left(-\frac{a^{-1}}{q-q^{-1}}\psi\right)&=\sum_{i=0}^d \left(\prod_{j=1}^i \frac{aq^{j-1}-a^{-1}q^{1-j}}{q^j-q^{-j}}\right)\psi^i,\\
\exp_{q}\left(\frac{a^{-1}}{q-q^{-1}}\psi\right)\exp_{q^{-1}}\left(-\frac{a}{q-q^{-1}}\psi\right)&=\sum_{i=0}^d \left(\prod_{j=1}^i \frac{a^{-1}q^{j-1}-aq^{1-j}}{q^j-q^{-j}}\right)\psi^i. 
\end{align*}
\end{cor}
\begin{proof}
Combine Theorem \ref{thm:Deltapoly} and Theorem \ref{thm:Deltaqexp}.
The equations can also be obtained directly by expanding their left-hand sides using \eqref{eq:qexpdef} and \eqref{eq:qexpinv}, and evaluating the results using the $q$-binomial theorem \cite[Theorem 10.2.1]{andrews}. 
\end{proof}

\section{The eigenvalues and eigenspaces of $\mathcal{M}$}\label{section:W}

We continue to discuss the situation of Assumption \ref{assump:main}.   
In Section \ref{section:M} we introduced the linear transformation $\mathcal{M}$.  Proposition \ref{prop:KexpM} indicates the role of $\mathcal{M}$ in the factorization of $\Delta$ in Theorem \ref{thm:Deltaqexp}.  In this section we show that $\mathcal{M}$ is diagonalizable.  We describe the eigenvalues and eigenspaces of $\mathcal{M}$.  We also explain how the eigenspace decomposition for $\mathcal{M}$ is related to the first and second split decompositions.

\begin{lemma}\label{lemma:Meigen}
The map $\mathcal{M} $ is diagonalizable with eigenvalues $q^{d},q^{d-2},q^{d-4},\mathellipsis,q^{-d}$.
\end{lemma}
\begin{proof}
Let $E=\exp_{q}(\frac{a^{-1}}{q-q^{-1}}\psi)$.  By \eqref{eq:KexpM}, $\mathcal{M}=E^{-1}KE$.  
By construction $K$ is diagonalizable with eigenvalues $q^{d},q^{d-2},q^{d-4},\mathellipsis,q^{-d}$.  
The result follows.
\end{proof}

\begin{definition}\label{def:W}{\rm 
For $0\leq i\leq d$, let $W_i$ denote the eigenspace of $\mathcal{M} $ corresponding to the eigenvalue $q^{d-2i}$.  Note that $\{W_i\}_{i=0}^d$ is a decomposition of $V$, and that  $W_i^\Downarrow=W_i$ for $0\leq i\leq d$.
   For notational convenience, let $W_{-1}=0$ and $W_{d+1}=0$.  
}\end{definition}

\begin{prop} \label{prop:UW}
For $0\leq i\leq d$, 
\begin{align}
&U_i=\exp_q\left(\frac{a^{-1}}{q-q^{-1}}\psi\right) W_i,&&U_i^\Downarrow=\exp_q\left(\frac{a}{q-q^{-1}}\psi\right) W_i, \label{eq:UexpW}\\
&W_i=\exp_{q^{-1}}\left(-\frac{a^{-1}}{q-q^{-1}}\psi\right) U_i, 
&&W_i=\exp_{q^{-1}}\left(-\frac{a}{q-q^{-1}}\psi\right) U_i^\Downarrow.\label{eq:WexpU}
\end{align}
\end{prop}
\begin{proof}
Define $E$ as in the proof of Lemma \ref{lemma:Meigen}. 
We show that $U_i=EW_i$. 
By \eqref{eq:KexpM},
$KE=E\mathcal{M}$.  
Recall that $U_i$ (resp. $W_i$) is the eigenspace of $K$ (resp. $\mathcal{M}$) corresponding to the eigenvalue $q^{d-2i}$.
By these comments $U_i=EW_i$.

Define $F=\exp_{q}(\frac{a}{q-q^{-1}}\psi)$. 
We show $U_i^\Downarrow=FW_i$.
By \eqref{eq:BexpM},
$BF=F\mathcal{M}$.  
Recall that $U_i^\Downarrow$ (resp. $W_i$) is the eigenspace of $B$ (resp. $\mathcal{M}$) corresponding to the eigenvalue $q^{d-2i}$.
By these comments $U_i^\Downarrow=FW_i$.

To obtain  \eqref{eq:WexpU} from  \eqref{eq:UexpW}, use \eqref{eq:qexpinv}.
\end{proof}

\begin{lemma}\label{lemma:dimW}
For $0\leq i\leq d$, the dimension of $W_i$ is $\rho_i$.
\end{lemma}
\begin{proof}
This follows from Proposition \ref{prop:UW} and the fact that $U_i,U_i^\Downarrow$ have dimension $\rho_i$.
\end{proof}

Recall from (\ref{eq:UUdd}) that
\begin{equation}
\sum_{h=0}^i E_h^*V=\sum_{h=0}^i U_h=\sum_{h=0}^i U_h^{\Downarrow} \label{eq:sum}
\end{equation}
for $0\leq i\leq d$.

\begin{lemma}\label{lemma:Wsum}
For $0\leq i\leq d$, the sum $\sum_{h=0}^iW_h$ is equal to the common value of \rm{(\ref{eq:sum})}.
\end{lemma}
\begin{proof}
Define $W=\sum_{h=0}^i W_h$ and let $U$ denote the common value of (\ref{eq:sum}).  We show that $W=U$.  By Lemma \ref{lemma:psiU} and the equation on the left in \eqref{eq:WexpU}, $W\subseteq U$. 
 By Lemma \ref{lemma:dimW}, $W$ and $U$ have the same dimension.  Thus $W=U$.
\end{proof}

\section{The actions of $\psi, K,B,\Delta,A,A^*$ on $\{W_i\}_{i=0}^d$}\label{section:actionW}

We continue to discuss the situation of Assumption \ref{assump:main}.   
Recall the eigenspace decomposition $\{W_i\}_{i=0}^d$ for  $\mathcal{M}$.
In this section, we discuss the actions of $\psi, K,B,\Delta,A,A^*$ on $\{W_i\}_{i=0}^d$.  

\begin{lemma}\label{lemma:psiW}
For $0\leq i\leq d$,
\begin{equation}
\psi W_i\subseteq W_{i-1}.
\end{equation}
\end{lemma}
\begin{proof}
Use Lemma \ref{lemma:Mpsi}.
\end{proof}

\begin{lemma}
For $0\leq i\leq d$,
\begin{align}
(K-q^{d-2i}I)W_i\subseteq W_{i-1},\qquad\qquad
(B-q^{d-2i}I)W_i\subseteq W_{i-1}.
\end{align}
\end{lemma}
\begin{proof}
Use Lemma \ref{lemma:K-M} and Lemma \ref{lemma:psiW}.
\end{proof}

\begin{lemma}
For $0\leq i\leq d$,
\begin{align}
(\Delta-I)W_i&\subseteq W_0+W_1+\cdots +W_{i-1},\label{eq:DeltaW1}\\
(\Delta^{-1}-I)W_i&\subseteq W_0+W_1+\cdots +W_{i-1}\label{eq:DeltaW2}.
\end{align}
\end{lemma}
\begin{proof}
To show (\ref{eq:DeltaW1}), use (\ref{eq:DeltaU22}) and Lemma \ref{lemma:Wsum}.

To show (\ref{eq:DeltaW2}), use \eqref{eq:Delta^{-1}U} and Lemma \ref{lemma:Wsum}.
\end{proof}

\begin{lemma}\label{lemma:AW}
For $0\leq i\leq d$,
\begin{align}
(A-(a+a^{-1})q^{d-2i}I)W_i\subseteq W_{i-1}+W_{i+1}.
\end{align}
\end{lemma}
\begin{proof}
By Lemma \ref{lemma:M^-2A}, the expression 
\begin{equation*}
(\mathcal{M}^{-1}-q^{2i+2-d}I)(\mathcal{M}^{-1}-q^{2i-2-d}I)(A-(a+a^{-1})q^{d-2i}I)
\end{equation*}
vanishes on $W_i$.  Therefore $(\mathcal{M}^{-1}-q^{2i+2-d}I)(\mathcal{M}^{-1}-q^{2i-2-d}I)$ vanishes on $(A-(a+a^{-1})q^{d-2i}I)W_i$.
The result follows.
\end{proof}

\begin{lemma}
For $0\leq i\leq d$,
\begin{align}
(A^*-\theta_i^*I)W_i\subseteq W_0+W_1+\cdots +W_{i-1}.
\end{align}
\end{lemma}
\begin{proof}
Use $(A^*-\theta_i^*I)E_i^*V=0$ together with \eqref{eq:DeltaU22} and Lemma \ref{lemma:Wsum}. 
\end{proof}

\section{The actions of $\mathcal{M}^{\pm 1}$ on $\{U_i\}_{i=0}^d$,$\{U_i^\Downarrow\}_{i=0}^d$, $\{E_iV\}_{i=0}^d$, $\{E_i^*V\}_{i=0}^d$}\label{section:Maction}

We continue to discuss the situation of Assumption \ref{assump:main}.   
In Section \ref{section:W} we saw how various operators act on the decomposition $\{W_i\}_{i=0}^d$.  In this section we investigate the action of $\mathcal{M} $ on the first and second split decompositions of $V$, as well as on the eigenspace decompositions of $A,A^*$.

\begin{lemma}\label{lemma:MU}
For $0\leq i\leq d$, 
\begin{align}
(\mathcal{M} -q^{d-2i}I)U_i&\subseteq U_0+U_1+\cdots+U_{i-1},\label{eq:MU1}\\
(\mathcal{M} -q^{d-2i}I)U_i^\Downarrow&\subseteq U_0^\Downarrow+U_1^\Downarrow+\cdots+U_{i-1}^\Downarrow.\label{eq:MU2}
\end{align}
\end{lemma}
\begin{proof}
To show (\ref{eq:MU1}),
use Definition \ref{def:KB}, Lemma \ref{lemma:KUdd}, and Definition \ref{def:M}.

To show (\ref{eq:MU2}), use (\ref{eq:MU1}) applied to $\Phi^\Downarrow$, along with $\mathcal{M}^\Downarrow=\mathcal{M}$.
\end{proof}

\begin{lemma}\label{lemma:MinvU}
For $0\leq i\leq d$, 
\begin{align}
(\mathcal{M}^{-1}-q^{2i-d}I)U_i\subseteq U_{i-1},\qquad\qquad (\mathcal{M}^{-1}-q^{2i-d}I)U_i^\Downarrow \subseteq U_{i-1}^\Downarrow.\label{eq:MinvU}
\end{align}
\end{lemma}
\begin{proof}
We first show the equation on the left in (\ref{eq:MinvU}). By Lemma \ref{lemma:Minv-exp}, 
\begin{equation}
\mathcal{M}^{-1}=(I-a^{-1}q^{-1}\psi)K^{-1}.
\end{equation}
From this and Definition \ref{def:KB}, it follows that on $U_i$,
\begin{equation}
\mathcal{M}^{-1}-q^{2i-d}I=a^{-1}q^{2i-d-1}\psi.
\end{equation}
The result follows from this along with Lemma \ref{lemma:psiU}.

The proof of the equation on the right in (\ref{eq:MinvU}) follows from the equation on the left in (\ref{eq:MinvU}) applied to $\Phi^\Downarrow$, along with the fact that $\mathcal{M}^\Downarrow=\mathcal{M}$.
\end{proof}

\begin{lemma}
For $0\leq i\leq d$,
\begin{equation}
\mathcal{M}^{-1}E_iV\subseteq E_{i-1}V+E_iV+E_{i+1}V.
\end{equation}
\end{lemma}
\begin{proof}
We first show that $\mathcal{M}^{-1}E_iV\subseteq \sum_{h=0}^{i+1}E_hV$.   
 Recall from (\ref{eq:EUdd}) that $E_iV\subseteq \sum_{h=d-i}^d U_h^\Downarrow$. 
 By this, Lemma \ref{lemma:MinvU}, and (\ref{eq:EUdd}), we obtain  $\mathcal{M}^{-1}E_iV\subseteq \sum_{h=0}^{i+1} E_hV$. 
 
We now show that $\mathcal{M}^{-1}E_iV\subseteq \sum_{h=i-1}^d E_hV$.  
 Recall from (\ref{eq:EU}) that $E_iV\subseteq \sum_{h=i}^d U_h$. 
By this, Lemma \ref{lemma:MinvU}, and (\ref{eq:EU}), we obtain 
 $\mathcal{M}^{-1}E_iV\subseteq \sum_{h=i-1}^d E_hV$. 
 
Thus $\mathcal{M}^{-1}E_iV$ is contained in the intersection of $\sum_{h=0}^{i+1} E_hV$ and $\sum_{h=i-1}^d E_hV$, which is $E_{i-1}V+E_iV+E_{i+1}V$.
\end{proof}

\begin{lemma}
For $0\leq i\leq d$, 
\begin{align*}
(\mathcal{M} -q^{d-2i}I)E_i^*V&\subseteq E_0^*V+E_1^*V+\cdots+E_{i-1}^*V,\\
(\mathcal{M}^{-1} -q^{2i-d}I)E_i^*V&\subseteq E_0^*V+E_1^*V+\cdots+E_{i-1}^*V.
\end{align*}
\end{lemma}
\begin{proof}
Note that $E_i^*V\subseteq E_0^*V+E_1^*V+\cdots+E_i^*V= W_0+W_1+\cdots+W_i$ by Lemma \ref{lemma:Wsum}.
The result follows from this fact along with Definition \ref{def:W}.
\end{proof}

\section{When $\Phi$ is a Leonard system}\label{section:LS}

We continue to discuss the situation of Assumption \ref{assump:main}.   
For the rest of the paper we assume $\rho_i=1$ for $0\leq i\leq d$.
In this case $\Phi$ is called a Leonard system.\\  

We use the following notational convention.  Let $\{v_i\}_{i=0}^d$ denote a basis for $V$.  The sequence of subspaces $\{\mathbb{K} v_i\}_{i=0}^d$ is a decomposition of $V$ said, to be \emph{induced} by the basis $\{v_i\}_{i=0}^d$.\\

We display a basis $\{u_i\}_{i=0}^d$ (resp. $\{u_i^\Downarrow\}_{i=0}^d$) (resp. $\{w_i\}_{i=0}^d$) that induces the decomposition $\{U_i\}_{i=0}^d$ (resp. $\{U_i^\Downarrow\}_{i=0}^d$ ) (resp. $\{W_i\}_{i=0}^d$ ).  We find the actions of $\psi,K,B,\Delta^{\pm 1}, A$ on these bases.  We also display the transition matrices between these bases. \\

For the rest of this section fix $0\neq u_0\in U_0$.  
Let $M$ denote the subalgebra of ${\rm End}(V)$ generated by $A$.  By \cite[Lemma 5.1]{LP24}, the map $M\to V$, $X\mapsto Xu_0$ is an isomorphism of vector spaces. 
Consequently, the vectors $\{A^iu_0\}_{i=0}^d$ form a basis for $V$.\\

We now define a basis $\{u_i\}_{i=0}^d$ of $V$ that induces $\{U_i\}_{i=0}^d$.  
 For $0\leq i\leq d$, define 
 \begin{equation}
 u_i=\left(\prod_{j=0}^{i-1}\left(A-\theta_jI\right)\right) u_0.\label{eq:ubasisdef}
 \end{equation}
Observe that $u_i\neq 0$. 
  By \cite[Theorem 4.6]{Somealg}, $u_i\in U_i$. 
   So $u_i$ is a basis for $U_i$.  Consequently, $\{u_i\}_{i=0}^d$ is a basis for $V$ that induces $\{U_i\}_{i=0}^d$.\\
   
Next we define a basis $\{u_i^\Downarrow\}_{i=0}^d$ of $V$ that induces $\{U_i^\Downarrow\}_{i=0}^d$.  For $0\leq i\leq d$, define
 \begin{equation}
 u_i^\Downarrow=\left(\prod_{j=0}^{i-1}\left(A-\theta_{d-j}I\right)\right) u_0.\label{eq:uddbasisdef}
 \end{equation}
Observe that  
 $u_i^\Downarrow\neq 0$.
By Lemma \ref{lemma:DD'}, $u_i^\Downarrow\in U_i^\Downarrow$.  So $u_i^\Downarrow$ is a basis for $U_i^\Downarrow$.   Consequently, $\{u_i^\Downarrow\}_{i=0}^d$ is a basis for $V$ that induces $\{U_i^\Downarrow\}_{i=0}^d$.\\

\begin{lemma}\label{lemma:uddDeltau}
 For $0\leq i\leq d$, 
	   \begin{equation}
	   u_i^\Downarrow=\Delta u_i.
	\end{equation}
\end{lemma}
\begin{proof}
By Lemma \ref{lemma:DD'}, $\Delta U_i=U_i^\Downarrow$.  So there exists $0\neq \lambda\in\mathbb{K}$ such that $\Delta u_i=\lambda u_i^\Downarrow$.  We show that $\lambda=1$.  By \cite[Lemma 7.3]{bockting1} and \eqref{eq:DeltaU22}, $\Delta u_i-A^iu$ is a linear combination of $\{A^ju\}_{j=0}^{i-1}$. Also, $u_i^\Downarrow-A^iu$ is a linear combination of $\{A^ju\}_{j=0}^{i-1}$.  The vectors $\{A^ju\}_{j=0}^{i-1}$ are linearly independent.  By these comments $\lambda=1$.
\end{proof}\\

We next define a basis $\{w_i\}_{i=0}^d$ of $V$ that induces $\{W_i\}_{i=0}^d$.  For $0\leq i\leq d$, define
 \begin{equation}
 w_i=\exp_{q^{-1}}\left(-\frac{a^{-1}}{q-q^{-1}}\ \psi\right) u_i.\label{eq:wdef}
\end{equation}
Since  $\{u_i \}_{i=0}^d$ is a basis of $V$ and $\exp_{q^{-1}}(-\frac{a^{-1}}{q-q^{-1}}\ \psi)$ is invertible, 
 $w_i$ is a basis for $W_i$.  
Consequently, $\{w_i\}_{i=0}^d$ is a basis for $V$ that induces $\{W_i\}_{i=0}^d$.

\begin{lemma}\label{lemma:uw}
For $0\leq i\leq d$, 
\begin{align}
&u_i=\exp_q\left(\frac{a^{-1}}{q-q^{-1}}\psi\right) w_i,&&u_i^\Downarrow=\exp_q\left(\frac{a}{q-q^{-1}}\psi\right) w_i,  \label{eq:uexpw}\\
&w_i=\exp_{q^{-1}}\left(-\frac{a^{-1}}{q-q^{-1}}\psi\right) u_i, 
&&w_i=\exp_{q^{-1}}\left(-\frac{a}{q-q^{-1}}\psi\right) u_i^\Downarrow.\label{eq:wexpu}
\end{align}
\end{lemma}
\begin{proof}
Use \eqref{eq:wdef} to obtain the equations on the left in \eqref{eq:uexpw},\eqref{eq:wexpu}.
 To obtain the equations on the right in \eqref{eq:uexpw},\eqref{eq:wexpu}, use  Theorem \ref{thm:Deltaqexp}, Lemma \ref{lemma:uddDeltau}, and \eqref{eq:wdef}. 
\end{proof}\\

We now describe the actions of $\psi,K,B,\mathcal{M}, \Delta,A$ on the bases
$\{u_i\}_{i=0}^d$, $\{u_i^\Downarrow\}_{i=0}^d$, $\{w_i\}_{i=0}^d$.  
First we recall a notion from linear algebra.  
Let ${\rm Mat}_{d+1}(\mathbb{K})$ denote the $\mathbb{K}$-algebra of $(d+1)\times(d+1)$ matrices that have all entries in $\mathbb{K}$. We index the rows and columns by $0,1,\mathellipsis,d$.  
Let $\{v_i\}_{i=0}^d$ denote a basis of $V$.  
For $T\in{\rm End}(V)$ and $X\in {\rm Mat}_{d+1}(\mathbb{K})$, we say that \emph{$X$ represents $T$} with respect to $\{v_i\}_{i=0}^d$ whenever $Tv_j=\sum_{i=0}^d X_{ij}v_i$ for $0\leq j\leq d$.  \\

By \eqref{eq:ubasisdef} and \eqref{eq:uddbasisdef},
the matrices that represent $A$ with respect to $\{u_i\}_{i=0}^d$ and $\{u_i^\Downarrow\}_{i=0}^d$ are, respectively, 
\begin{equation}
\displaystyle \left(\begin{matrix}
\theta_0& &  &{\bf 0}\\
1 & \theta_1 & &\\
&\ddots&\ddots&\\
{\bf 0}&&1&\theta_d
\end{matrix}\right), 
\qquad\qquad\qquad
\displaystyle \left(\begin{matrix}
\theta_d& &  &{\bf 0}\\
1 & \theta_{d-1} & &\\
&\ddots&\ddots&\\
{\bf 0}&&1&\theta_0
\end{matrix}\right). \label{eq:Amatrices}
\end{equation}

By construction, the matrix ${\rm diag}(q^d,q^{d-2},\mathellipsis,q^{-d})$ represents $K$ with respect to $\{u_i\}_{i=0}^d$, and $B$ with respect to $\{u_i^\Downarrow\}_{i=0}^d$, and $\mathcal{M}$ with respect to $\{w_i\}_{i=0}^d$.

\begin{definition}\label{def:psihat}
We define a matrix
$\widehat{\psi}\in {\rm Mat}_{d+1}(\mathbb{K})$.  For $1\leq i\leq d$, the $(i-1,i)$-entry is $(q^i-q^{-i})(q^{d-i+1}-q^{i-d-1})$.  All other entries are 0. 
\end{definition}

\begin{prop}\label{prop:matpsi}
The matrix $\widehat\psi$ represents $\psi$ with respect to each of the bases $\{u_i\}_{i=0}^d$, $\{u_i^\Downarrow\}_{i=0}^d$, $\{w_i\}_{i=0}^d$.
\end{prop}
\begin{proof}
By  \cite[Line (23)]{bockting2}, $\widehat\psi$ represents $\psi$ with respect to $\{u_i\}_{i=0}^d$.  
The remaining assertions follow from Lemma \ref{lemma:uw}.
\end{proof}\\

Next we give the matrices that represent $\mathcal{M}^{\pm 1}$ with respect to the bases $\{u_i\}_{i=0}^d$, $\{u_i^\Downarrow\}_{i=0}^d$. 

\begin{lemma}
We give the matrix in ${\rm Mat}_{d+1}(\mathbb{K})$ that represents $\mathcal{M}$ with respect to $\{u_i\}_{i=0}^d$.  This matrix is upper triangular.  For $0\leq i\leq j\leq d$, the $(i,j)$-entry is
\begin{equation}
a^{i-j}q^{d-j-i}\left(q-q^{-1}\right)^{2(j-i)} \frac{\ \ [j]_q^![d-i]_q^!\ \ }{[i]_q^![d-j]_q^!}.
\end{equation}
\end{lemma}
\begin{proof}
The matrix ${\rm diag}(q^d,q^{d-2},\mathellipsis,q^{-d})$ represents $K$ with respect to $\{u_i\}_{i=0}^d$.
Use this fact along with Lemma \ref{lemma:Msum} and Proposition  \ref{prop:matpsi}.
\end{proof}

\begin{lemma}
We give the matrix in ${\rm Mat}_{d+1}(\mathbb{K})$ that represents $\mathcal{M}^{-1}$ with respect to $\{u_i\}_{i=0}^d$.  
For $0\leq i\leq d$, the $(i,i)$-entry is $q^{2i-d}$.  For $1\leq i\leq d$, the $(i-1,i)$-entry is $$-a^{-1}q^{2i-d-1}\left(q^i-q^{-i}\right)\left(q^{d-i+1}-q^{i-d-1}\right).$$  All other entries are zero.
\end{lemma}
\begin{proof}
The matrix ${\rm diag}(q^{-d},q^{2-d},\mathellipsis,q^d)$ represents $K^{-1}$ with respect to $\{u_i\}_{i=0}^d$.
 Use this fact along with Lemma \ref{lemma:Minv-exp} and Proposition  \ref{prop:matpsi}.
\end{proof}

\begin{lemma}
We give the matrix in ${\rm Mat}_{d+1}(\mathbb{K})$ that represents $\mathcal{M}$ with respect to $\{u_i^\Downarrow\}_{i=0}^d$.  This matrix is upper triangular.  For $0\leq i\leq j\leq d$, the $(i,j)$-entry is
\begin{equation}
a^{j-i}q^{d-j-i}\left(q-q^{-1}\right)^{2(j-i)} \frac{\ \ [j]_q^![d-i]_q^!\ \ }{[i]_q^![d-j]_q^!}.
\end{equation}
\end{lemma}
\begin{proof}
The matrix ${\rm diag}(q^d,q^{d-2},\mathellipsis,q^{-d})$ represents $B$ with respect to $\{u_i^\Downarrow\}_{i=0}^d$.
 Use this fact along with Lemma \ref{lemma:Msum} and Proposition  \ref{prop:matpsi}.
\end{proof}

\begin{lemma}
We give the matrix in ${\rm Mat}_{d+1}(\mathbb{K})$ that represents $\mathcal{M}^{-1}$ with respect to $\{u_i^\Downarrow\}_{i=0}^d$.  
For $0\leq i\leq d$, the $(i,i)$-entry is $q^{2i-d}$.  For $1\leq i\leq d$, the $(i-1,i)$-entry is $$-aq^{2i-d-1}\left(q^i-q^{-i}\right)\left(q^{d-i+1}-q^{i-d-1}\right).$$  All other entries are zero.
\end{lemma}
\begin{proof}
The matrix ${\rm diag}(q^{-d},q^{2-d},\mathellipsis,q^d)$ represents $B^{-1}$ with respect to $\{u_i^\Downarrow\}_{i=0}^d$.
 Use this fact along with Lemma \ref{lemma:Minv-exp} and Proposition  \ref{prop:matpsi}.
\end{proof}

Next we give the matrices that represent $K$ with respect to the bases $\{u_i^\Downarrow\}_{i=0}^d$, $\{w_i\}_{i=0}^d$.

\begin{lemma}
We give the matrix in ${\rm Mat}_{d+1}(\mathbb{K})$ that represents $K$ with respect to $\{u_i^\Downarrow\}_{i=0}^d$.  
For $0\leq i\leq d$, the $(i,i)$-entry is $q^{d-2i}$.  
 For $0\leq i<j\leq d$, the $(i,j)$-entry is  
\begin{equation}
\left(1-a^{-2}\right)a^{j-i}q^{d-j-i} \left(q-q^{-1}\right)^{2(j-i)} \frac{\ \ [j]_q^![d-i]_q^!\ \ }{[i]_q^![d-j]_q^!}. 
\end{equation}
All other entries are zero.
\end{lemma}
\begin{proof}
Evaluating the equation on the right in \eqref{eq:psiequations1} 
using the equation on the left in \eqref{eq:invertible1a}  
we get
\begin{equation}
K=\left(a^{-2}I+(1-a^{-2})\sum_{n=0}^d a^nq^{n}\psi^n\right)B. \label{eq:KKBB2-1}
\end{equation}
The result follows from this along with Proposition \ref{prop:matpsi} and the fact that 
the matrix ${\rm diag}(q^d,q^{d-2},\mathellipsis,q^{-d})$ represents $B$ with respect to $\{u_i^\Downarrow\}_{i=0}^d$.
\end{proof}

\begin{lemma}\label{lemma:matKw}
We give the matrix in ${\rm Mat}_{d+1}(\mathbb{K})$ that represents $K$ with respect to $\{w_i\}_{i=0}^d$.  
For $0\leq i\leq d$, the $(i,i)$-entry is $q^{d-2i}$.  For $1\leq i\leq d$, the $(i-1,i)$-entry is $$-a^{-1}q^{d-2i+1} (q^{i}-q^{-i})(q^{d-i+1}-q^{i-d-1}).$$  All other entries are zero.
\end{lemma}
\begin{proof}
The matrix ${\rm diag}(q^d,q^{d-2},\mathellipsis,q^{-d})$ represents $\mathcal{M}$ with respect to $\{w_i\}_{i=0}^d$.
	Use this fact along with Proposition \ref{prop:matpsi} and the equation on the left in (\ref{eq:K-M}).
\end{proof}

Next we give the matrices that represent $B$ with respect to the bases $\{u_i\}_{i=0}^d$, $\{w_i\}_{i=0}^d$.

\begin{lemma}
We give the matrix in ${\rm Mat}_{d+1}(\mathbb{K})$ that represents $B$ with respect to $\{u_i\}_{i=0}^d$.  
For $0\leq i\leq d$, the $(i,i)$-entry is $q^{d-2i}$.  
 For $0\leq i<j\leq d$, the $(i,j)$-entry is  
\begin{equation}
\left(1-a^{2}\right)a^{i-j}q^{d-j-i}\left(q-q^{-1}\right)^{2(j-i)} \frac{\ \ [j]_q^![d-i]_q^!\ \ }{[i]_q^![d-j]_q^!}.
\end{equation}
All other entries are zero.
\end{lemma}
\begin{proof}
Evaluating the equation on the left in \eqref{eq:psiequations1}  
using the equation on the right in \eqref{eq:invertible1a}  
we get
\begin{equation}
B=\left(a^{2}I+(1-a^{2})\sum_{n=0}^d a^{-n}q^{n}\psi^n\right)K. \label{eq:KKBB2-2}
\end{equation}
The result follows from this along with Proposition \ref{prop:matpsi} and the fact that 
the matrix ${\rm diag}(q^d,q^{d-2},\mathellipsis,q^{-d})$ represents $K$ with respect to $\{u_i\}_{i=0}^d$.
\end{proof}

\begin{lemma} 
We give the matrix in ${\rm Mat}_{d+1}(\mathbb{K})$ that represents $B$ with respect to $\{w_i\}_{i=0}^d$.  
For $0\leq i\leq d$, the $(i,i)$-entry is $q^{d-2i}$.  For $1\leq i\leq d$, the $(i-1,i)$-entry is $$-aq^{d-2i+1} (q^{i}-q^{-i})(q^{d-i+1}-q^{i-d-1}).$$  All other entries are zero.
\end{lemma}
\begin{proof}
The matrix ${\rm diag}(q^d,q^{d-2},\mathellipsis,q^{-d})$ represents $\mathcal{M}$ with respect to $\{w_i\}_{i=0}^d$.
	Use this fact along with Proposition \ref{prop:matpsi} and the equation on the left in (\ref{eq:B-M}). 
	\end{proof}\\

Next we consider the matrices 
\begin{equation}
\exp_q\left(\frac{a}{q-q^{-1}}\widehat{\psi}\right),\qquad\qquad \exp_q\left(\frac{a^{-1}}{q-q^{-1}}\widehat{\psi}\right).\label{eq:expmats1}
\end{equation}
Their inverses are
\begin{equation}
\exp_{q^{-1}}\left(-\frac{a}{q-q^{-1}}\widehat{\psi}\right),\qquad\qquad \exp_{q^{-1}}\left(-\frac{a^{-1}}{q-q^{-1}}\widehat{\psi}\right)\label{eq:expmats2}
\end{equation}
respectively.  
The matrices in \eqref{eq:expmats1},\eqref{eq:expmats2}
are upper triangular.  We now consider the entries of \eqref{eq:expmats1},\eqref{eq:expmats2}.\\

\begin{lemma}
For $0\neq x\in\mathbb{K}$, the matrix $\exp_q(x\widehat{\psi})$ is upper triangular.  For $0\leq i\leq j\leq d$, the $(i,j)$-entry is
\begin{equation}
x^{j-i}q^{\binom{j-i}{2}}\left(q-q^{-1}\right)^{2(j-i)}\cdot \frac{ [j]_q^![d-i]_q^! }{[i]_q^![j-i]_q^![d-j]_q^!}.
\end{equation}
The matrix $exp_{q^{-1}}(x\widehat{\psi})$ is upper triangular.  For $0\leq i\leq j\leq d$, the $(i,j)$-entry is
\begin{equation}
x^{j-i}q^{-\binom{j-i}{2}}\left(q-q^{-1}\right)^{2(j-i)}\cdot \frac{ [j]_q^![d-i]_q^! }{[i]_q^![j-i]_q^![d-j]_q^!}.
\end{equation}
\end{lemma}

\begin{lemma}\label{lemma:trans1}
The transition matrices between the basis  $\{w_i\}_{i=0}^d$ and the bases $\{u_i\}_{i=0}^d$, $\{u_i^\Downarrow\}_{i=0}^d$ are given in the table below. 
\begin{center}
\begin{tabular}{ c |c |c }
 \text{\ \ \ From \ \ \ } & \text{\ \ \ To \ \ \ } & \text{\ Transition Matrix\ }\\
 \hline
 $\{u_i\}_{i=0}^d$& $\{w_i\}_{i=0}^d$ &  $\exp_{q^{-1}}\left(-\frac{a^{-1}}{q-q^{-1}}\widehat\psi\right)$\\
 $\{w_i\}_{i=0}^d$ &$ \{u_i\}_{i=0}^d$ & $\exp_q\left(\frac{a^{-1}}{q-q^{-1}}\widehat\psi\right)$\\
 $\{u_i^\Downarrow\}_{i=0}^d$ & $\{w_i\}_{i=0}^d$ & $\exp_{q^{-1}}\left(-\frac{a}{q-q^{-1}}\widehat\psi\right)$\\
  $\{w_i\}_{i=0}^d$ & $\{u_i^\Downarrow\}_{i=0}^d$ & $\exp_q\left(\frac{a}{q-q^{-1}}\widehat\psi\right)$\\
 \end{tabular}
\end{center}
\end{lemma}
\begin{proof}
Use Lemma \ref{lemma:uw} and Proposition \ref{prop:matpsi}.
\end{proof}

We next consider the product 
\begin{equation}
\exp_q\left(\frac{a}{q-q^{-1}}\widehat\psi\right)\exp_{q^{-1}}\left(-\frac{a^{-1}}{q-q^{-1}}\widehat\psi\right).\label{eq:matpsiexp}
\end{equation}
The inverse of \eqref{eq:matpsiexp} is 
\begin{equation}
\exp_q\left(\frac{a^{-1}}{q-q^{-1}}\widehat\psi\right)\exp_{q^{-1}}\left(-\frac{a}{q-q^{-1}}\widehat\psi\right).
\label{eq:matpsiexpinv}
\end{equation}
The matrices in \eqref{eq:matpsiexp},\eqref{eq:matpsiexpinv}
are upper triangular. \\

\begin{lemma}
The transition matrices between the bases $\{u_i\}_{i=0}^d$, $\{u_i^\Downarrow\}_{i=0}^d$ are given in the table below. 
\begin{center}
\begin{tabular}{ c |c |c }
 \text{\ \ \ From \ \ \ } & \text{\ \ \ To \ \ \ } & \text{\ Transition Matrix\ }\\
 \hline
 $\{u_i\}_{i=0}^d$ & $\{u^\Downarrow_i\}_{i=0}^d$ & $\exp_q\left(\frac{a}{q-q^{-1}}\widehat\psi\right)\exp_{q^{-1}}\left(-\frac{a^{-1}}{q-q^{-1}}\widehat\psi\right) $\\
 $\{u^\Downarrow_i\}_{i=0}^d$ & $\{u_i\}_{i=0}^d$ & $\exp_q\left(\frac{a^{-1}}{q-q^{-1}}\widehat\psi\right)\exp_{q^{-1}}\left(-\frac{a}{q-q^{-1}}\widehat\psi\right)$ 
 \end{tabular}
\end{center}
\end{lemma}
\begin{proof}
Use Lemma \ref{lemma:trans1}.
\end{proof}

\begin{lemma}\label{lemma:DeltaMat}
With respect to each of the bases $\{u_i\}_{i=0}^d, \{u_i^\Downarrow\}_{i=0}^d, \{w_i\}_{i=0}^d$, the matrices that represent $\Delta$ and $\Delta^{-1}$ are $\exp_q\left(\frac{a}{q-q^{-1}}\widehat\psi\right)\exp_{q^{-1}}\left(-\frac{a^{-1}}{q-q^{-1}}\widehat\psi\right)$
and
$\exp_q\left(\frac{a^{-1}}{q-q^{-1}}\widehat\psi\right)\exp_{q^{-1}}\left(-\frac{a}{q-q^{-1}}\widehat\psi\right)$
respectively.
\end{lemma}
\begin{proof}
	Use Theorem \ref{thm:Deltapoly} and Proposition \ref{prop:matpsi}.
\end{proof}
\medskip

We give the entries of the matrices representing $\Delta,\Delta^{-1}$ in the following lemma. 

\begin{lemma}
The matrix in \eqref{eq:matpsiexp} is upper triangular. 
For $0\leq i\leq j\leq d$, the $(i,j)$-entry of \eqref{eq:matpsiexp} is
\begin{equation}
 \frac{\left(q-q^{-1}\right)^{j-i} [j]_q^![d-i]_q^! }{[i]_q^![j-i]_q^![d-j]_q^! } \ \prod_{n=1}^{j-i} \left(aq^{n-1}-a^{-1}q^{1-n}\right).
\end{equation}
The matrix in \eqref{eq:matpsiexpinv} is upper triangular. 
For $0\leq i\leq j\leq d$, the $(i,j)$-entry of \eqref{eq:matpsiexpinv} is
\begin{equation}
\frac{\left(q-q^{-1}\right)^{j-i} [j]_q^![d-i]_q^! }{[i]_q^![j-i]_q^![d-j]_q^! } \ \prod_{n=1}^{j-i} \left(a^{-1}q^{n-1}-aq^{1-n}\right).
\end{equation}
\end{lemma}
\begin{proof}
Use Corollary \ref{cor:prodsum} and Proposition \ref{prop:matpsi}.
\end{proof}

We finish the paper by giving the matrix that represents $A$ with respect to $\{w_i\}_{i=0}^d$.

\begin{lemma}
We give the matrix in ${\rm Mat}_{d+1}(\mathbb{K})$ that represents $A$ with respect to $\{w_i\}_{i=0}^d$.  
For $1\leq i\leq d$, the $(i,i-1)$-entry is 1.
For $0\leq i\leq d$, the $(i,i)$-entry is $(a+a^{-1})q^{d-2i}$.  For $1\leq i\leq d$, the $(i-1,i)$-entry is $$-q^{d-2i+1}(q^i-q^{-i})(q^{d-i+1}-q^{i-d-1}).$$  All other entries are zero.
\end{lemma}
\begin{proof} Let $\mathcal{A}$ denote the matrix that represents $A$ with respect to $\{w_i\}_{i=0}^d$.  
By Lemma \ref{lemma:AW}, $\mathcal{A}$ is tridiagonal with $(i,i)$-entry 
given by $(a+a^{-1})q^{d-2i}$ for $0\leq i\leq d$. \smallskip

We now show that the subdiagonal entires of $\mathcal{A}$ are all 1.  Let $\mathcal{A'}$ denote the matrix that represents $A$ with respect to $\{u_i\}_{i=0}^d$.  Recall that this matrix is displayed on the left in \eqref{eq:Amatrices}.  
Observe that $\mathcal{A}$ is equal to $\exp_{q^{-1}}(-\frac{a^{-1}}{q-q^{-1}}\widehat\psi) \mathcal{A}' \exp_{q}(\frac{a^{-1}}{q-q^{-1}}\widehat\psi)$.  It follows from this fact that the subdiagonal entries of $\mathcal{A}$ are all 1. \smallskip

We next obtain the superdiagonal entries of $\mathcal{A}$.  Let $0\leq i\leq d$.  Apply both sides of  \eqref{eq:qAM^} to $w_i$.  Evaluate the result using Proposition \ref{prop:matpsi} and the fact that the $w_i$ is an eigenvector for $\mathcal{M}$ with eigenvalue $q^{2i-d}$. Analyze the result in light of the above comments concerning the entries of $\mathcal{A}$ to obtain the desired result. 
\end{proof}


\bibliographystyle{amsplain}
\bibliography{master}

\bigskip
\bigskip
\noindent Sarah Bockting-Conrad \hfil\break
\noindent Department of Mathematical Sciences \hfil\break
\noindent DePaul University \hfil\break
\noindent 2320 N. Kenmore Avenue\hfil\break
\noindent Chicago, IL 60614  USA \hfil\break
\noindent email: {\tt sarah.bockting@depaul.edu }\hfil\break



\end{document}